\documentclass{article}

\usepackage{amssymb,amsmath,color}
\usepackage{amsthm,setspace}
\usepackage{url}
\usepackage{tikz,subcaption}
\usepackage[enableskew]{youngtab}
\usepackage{ytableau,varwidth}

\definecolor{red}{rgb}{1,0,0}
\definecolor{blue}{rgb}{.2,.2,.8}

\def\S{\mathcal S}

\def\B{\mathcal B}
\def\A{\mathcal A}

\def\P{\mathcal P} 

\def\Q{\mathcal Q}
\def\M{\mathcal M}
\def\p{\mathcal{POD}}

\newtheorem{theorem}{Theorem}[section]
\newtheorem{corollary}[theorem]{Corollary}

\newtheorem{conjecture}{Conjecture}

\newtheorem{lemma}[theorem]{Lemma}

\theoremstyle{definition}
\newtheorem{definition}{Definition}

\newtheorem*{remark}{Remark}

\begin{document}

	\title{$4$-Regular partitions and the pod function}
	\author{Cristina Ballantine
	\\
	\footnotesize Department of Mathematics and Computer Science\\
	\footnotesize College of The Holy Cross\\
	\footnotesize Worcester, MA 01610, USA \\
	\footnotesize cballant@holycross.edu
	\and Mircea Merca
	\\ 
	\footnotesize Department of Mathematics\\
	\footnotesize University of Craiova\\
	\footnotesize 200585 Craiova, Romania\\
	\footnotesize mircea.merca@profinfo.edu.ro
}
	\date{}
	\maketitle


\begin{abstract}
The partition function $pod(n)$ enumerates the partitions of $n$ wherein odd
parts are distinct and even parts are unrestricted. Recently, a number of properties for $pod(n)$ have been established. In this paper, for $k\in\{0,2\}$ we consider the partitions of $n$ into distinct parts not congruent to $k$ modulo $4$ and the $4$-regular partitions of $n$ in order to obtain new properties for $pod(n)$. In this context, we derive two new infinite families of linear inequalities involving the function $pod(n)$ and obtain new identities of Watson type.
\end{abstract}

\section{Introduction}

A partition $\lambda$ of $n$ is a nonincreasing sequence $(\lambda_1, \lambda_2, \ldots, \lambda_\ell)$ of positive integers (called parts) such that $\sum_{i=1}^\ell\lambda_i=n$.  We refer to $\lambda_i$ as the $i$th part of $\lambda$, and as usual denote by $p(n)$  the number of integer partitions of $n$.  We note that $p(x)=0$ if $x \not \in \mathbb Z_{\geqslant 0}$, and since the empty partition $\emptyset$ is the only partition of $0$, we have that $p(0)=1$. The generating function for $p(n)$ satisfies the identity
\begin{align}
\sum_{n=0}^\infty p(n)\, q^n = \frac{1}{(q;q)_\infty}, \label{gfp}
\end{align}
where we use the  customary $q$-series notation
\begin{align*}
& (a;q)_n = \begin{cases}
1, & \text{for $n=0$,}\\
(1-a)(1-aq)\cdots(1-aq^{n-1}), &\text{for $n>0$;}
\end{cases}\\
& (a;q)_\infty = \lim_{n\to\infty} (a;q)_n.
\end{align*}
Moreover, we use the short notation
$$
(a_1,a_2,\ldots,a_n;q)_\infty = (a_1;q)_\infty (a_2;q)_\infty \cdots (a_n;q)_\infty.
$$
Because the infinite product $(a;q)_{\infty}$ diverges when $a\neq 0$ and $|q| \geqslant 1$, whenever
$(a;q)_{\infty}$ appears in a formula, we shall assume $|q| < 1$.

We denote by $pod(n)$ the function which enumerates the partitions of $n$ with odd
parts  distinct and even parts  unrestricted. Elementary techniques in the theory of partitions  give the following equivalent expressions for the  generating function for $pod(n)$:
\begin{equation}\label{podgen}
\sum_{n=0}^\infty pod(n)\, q^n = \frac{(-q;q^2)}{(q^2;q^2)_\infty}= \frac{(q^2;q^4)}{(q;q)_\infty} .
\end{equation}
The  $pod(n)$ function has been studied widely. It appears, for example, in the works of 
K. Alladi \cite{Alladi, Alladi16},
G. E. Andrews \cite{Andrews67,Andrews72}, 
G. E. Andrews and M. Merca \cite{Andrews18},
C. Ballantine, M. Merca, D. Passary, A. J. Yee \cite{BMPY},
A. Berkovich and F. Garvan \cite{Berkovich},
S.-P. Cui, W. X. Gu,  Z. S. Ma \cite{Cui15},
H. Fang, F. Xue, O. X. M. Yao \cite{Fang}, 
M. D. Hirschhorn and J. A. Sellers \cite{Hirschhorn10},
S. Radu and J. A. Sellers \cite{Radu}.
In this article, we consider the   interpretation of the $pod(n)$ function given by the last expression in \eqref{podgen} as the number of partitions of $n$ into parts not congruent to $2$ modulo $4$.

\begin{definition}
	Let $n$ be a nonnegative integer. We define 	$pod_{e}(n)$ (respectively $pod_o(n)$) to be the number of partitions of $n$ into an even (respectively odd) number of parts which are not congruent to $2$ modulo $4$.

\end{definition}

For example, since  the partitions of $8$ into parts not congruent to $2$ modulo $4$ are 
\begin{align*}
& (8),\ (7,1),\ (5,3),\ (5,1,1,1),\ (4,4),\ (4,3,1),\ (4,1,1,1,1),\\
& (3,3,1,1),\ (3,1,1,1,1,1),\ (1,1,1,1,1,1,1,1),
\end{align*}
we have that $pod_e(8)=7$ and $pod_o(8)=3$.

For an integer $\ell>1$,  a partition is called $\ell$-regular if none of its parts is divisible by $\ell$. In classical representation theory, $\ell$-regular partitions of $n$ parameterize the irreducible $\ell$-modular representations of the symmetric group $\S_n$ when $\ell$ is prime \cite{James}. The arithmetic properties of the  number $b_\ell(n)$ of  $\ell$-regular partitions of $n$
have been  investigated extensively (see, for example, \cite{Carlson,Cui,Dand,Furcy,Hirschhorn,Lovejoy,Penn, Penn8,Xia,Xia14,Webb}).   
The generating function for $b_{\ell}(n)$ satisfies the identity
$$
\sum_{n=0}^\infty b_\ell(n)\, q^n= \frac{(q^\ell;q^\ell)_\infty}{(q;q)_\infty}.
$$

In this article, we consider the $4$-regular partitions of $n$ and provide connections to the partitions of $n$ into distinct parts not congruent to $2$ modulo $4$. 

\begin{definition}
	Let $n$ be a nonnegative integer. We define
 $b_{4,e}(n)$ (respectively   $b_{4,o}(n)$) to be the number of $4$-regular partitions of $n$ into an even (respectively odd) number of parts.
	\end{definition} 


For example,  the partitions of $7$ into parts that are not multiples of $4$ are 
\begin{align*}
& (7),\ (6,1),\ (5,2),\ (5,1,1),\ (3,3,1),\ (3,2,2),\ (3,2,1,1),\ (3,1,1,1,1),\\
& (2,2,2,1),\ (2,2,1,1,1),\ (2,1,1,1,1,1),\ (1,1,1,1,1,1,1),
\end{align*}
and so  $b_4(7)=12$, $b_{4,e}(7)=5$ and $b_{4,o}(7)=7$.
The sequences $b_{4,e}(n)$ and $b_{4,o}(n)$  can be found in the
On-Line Encyclopedia of Integer Sequences \cite[A339406, A339407]{Sloane}.

\begin{definition}
	Let $n$ be a nonnegative integer and $k\in\{0,2\}$. We define $Q_k(n)$ to be the number of partitions of $n$ into distinct parts which are not congruent to $k$ modulo $4$.
\end{definition}

For example, $Q_0(14)=11$ since the $4$-regular partitions of $14$ into distinct parts are 
\begin{align*}
& (14),\ (13,1),\ (11,3),\ (11,2,1),\ (10,3,1),\ (9,5),\\ 
& (9,3,2),\ (7,6,1),\ (7,5,2),\ (6,5,3),\ (6,5,2,1),  
\end{align*}
while $Q_2(14) = 6$, the relevant partitions being 
\begin{align*}
(13,1),\ (11,3),\ (9,5),\ (9,4,1),\ (8,5,1),\ (7,4,3).
\end{align*}
The  generating functions for these sequences satisfy the identities
\begin{align}
\sum_{n=0}^\infty Q_0(n)\, q^n = (-q,-q^2,-q^3;q^4)_\infty \label{gfQ0}
\end{align}
and
\begin{align}
\sum_{n=0}^\infty Q_2(n)\, q^n = (-q,-q^3,-q^4;q^4)_\infty. \label{gfQ1}
\end{align}
We remark that the sequence $Q_0(n)$  can be found in the
On-Line Encyclopedia of Integer Sequences \cite[A070048]{Sloane}.  
Also on the page for $A070048$, we find another combinatorial interpretation for $Q_0(n)$: 
the number of partitions of $n$ into odd parts in which no part appears more than thrice.

  The following result introduces new combinatorial interpretations for the partition functions $Q_0(n)$ and $Q_2(n)$.

 \begin{theorem}\label{T:1}
 	For $n\geqslant 0$ the following hold.
 	\begin{enumerate}
 		\item[(i)]  $(-1)^n\, Q_0(n) =  pod_{e}(n)-pod_{o}(n) $
 		\item[(ii)] $(-1)^n\, Q_2(n) =  b_{4,e}(n)-b_{4,o}(n) $
 	\end{enumerate}
 \end{theorem}


 \begin{corollary}\label{CT:1}
	Let $n\geqslant 0$.
	\begin{enumerate}
		\item[(i)]  $Q_0(n)$ and $pod(n)$ have the same parity.
		\item[(ii)] $Q_2(n)$ and $b_{4}(n)$ have the same parity.
	\end{enumerate}
\end{corollary}

We also have the following result relating  $Q_0(n)$ and $Q_2(n)$ to $pod(n)$.
  
\begin{theorem}\label{T:2}
	For $n\geqslant 0$ the following hold.
	\begin{enumerate}
		\item[(i)] $\displaystyle{Q_0(n) = pod(n)+2 \sum_{k=1}^{\infty} (-1)^{k}\, pod\big(n-4k^2\big)}$
		\item[(ii)] $\displaystyle{Q_2(n) = \sum_{k=0}^{\infty} (-1)^{k(k+1)/2}\, pod\big(n-k(k+1)\big)}$
	\end{enumerate}
\end{theorem}

In addition,  $Q_0(n)$ and $Q_2(n)$ satisfy similar linear recurrence relations involving the triangular numbers. To make these easier to state, we first define $$
\xi_n = 
\begin{cases}
(-1)^{m}\cdot 2 & \text{if $n=4m^2$ for some $m>0$,}\\
1 & \text{if $n=0$,}\\
0 & \text{otherwise,}
\end{cases}
$$
and
$$
\chi_n =
\begin{cases}
(-1)^{n/2} & \text{if $n=m(m+1)$ for some $m\geqslant 0$,}\\
0 & \text{otherwise.}
\end{cases}
$$
\begin{theorem}\label{T:3}
	For $n\geqslant 0$ the following hold.
	\begin{enumerate}
		\item[(i)] $\displaystyle{\sum_{k=0}^\infty (-1)^{k(k+1)/2}\, Q_0\big(n-k(k+1)/2\big)
			= \xi_n}$
		\item[(ii)] $\displaystyle{\sum_{k=0}^\infty (-1)^{k(k+1)/2}\, Q_2\big(n-k(k+1)/2\big)
			=\chi_n}$
	\end{enumerate}
\end{theorem}

The rest of this paper is organized as follows. In Section \ref{S2}, we provide analytic and combinatorial proofs of Theorem \ref{T:1}. 
In Section \ref{S3}, we   provide
proofs of Theorems \ref {T:2} and \ref{T:3} using  generating functions and also give a combinatorial proof of Theorem \ref{T:2} (ii) (a combinatorial proof of Theorem \ref{T:2}(i)  would be very welcome). In Section \ref{S4}, we  show that the identities of Theorem \ref{T:2} are limiting cases of much more general identities and use the latter to derive two infinite families of linear inequalities involving $pod(n)$. In Section \ref{RC}, we present   several Ramanujan type congruences for $b_4(n)$ modulo $16$ and $64$. In  Section \ref{S5}, we  obtain three identities of Watson type and prove them analytically and combinatorially, and  conclude  with conjectures for two infinite families of linear inequalities involving $Q_0(n)$ and $Q_2(n)$.

\section{Proof of Theorem \ref{T:1}}
\label{S2}

\subsection{Analytic Proof}

Define
\begin{align*}
F(z,q) := \prod_{k=0}^\infty \frac{1}{(1-zq^{4k+1})(1-zq^{4k+3})(1-zq^{4k+4})}. \end{align*}
Then,   using Euler's identity 
	$1/(q;q^2)_\infty=(-q;q)_\infty$ and \eqref{gfQ0}, we have 

	\begin{align*}
 F(-1,q)&
	= \frac{1}{(-q,-q^3,-q^4;q^4)_\infty}  
 	= \frac{(-q^2;q^4)_\infty}{(-q;q)_\infty} \\ & 
	= (q;q^2)_\infty(-q^2;q^4)_\infty 
	= \sum_{n=0}^\infty (-1)^n\, Q_0(n)\, q^n.
	\end{align*}
	On the other hand, the fact that  $$F(z,q)= \sum_{m=0}^\infty \sum_{n=0}^\infty pod(n,m)\, z^m\, q^n, $$ where $pod(n,m)$ is equal to the number of partitions of $n$ with $m$ parts none of which are  congruent to $2$ modulo $4$, implies that $$F(-1,q)= \sum_{n=0}^\infty \big( pod_{e}(n)-pod_o(n) \big)\, q^n,$$ which establishes (i).
	To prove (ii), we define
$$
G(z,q) := \prod_{k=0}^\infty \frac{1}{(1-zq^{4k+1})(1-zq^{4k+2})(1-zq^{4k+3})}.$$ Then,   using Euler's identity 
	and \eqref{gfQ1}, we have 
\begin{align*}
G(-1,q) & 
= \frac{1}{(-q,-q^2,-q^3;q^4)_\infty}   = \frac{(-q^4;q^4)_\infty}{(-q;q)_\infty} \\ & 
= (q;q^2)_\infty(-q^4;q^4)_\infty 
= \sum_{n=0}^\infty (-1)^n\, Q_2(n)\, q^n.
\end{align*}
Moreover, since
$$G(z,q) = \sum_{m=0}^\infty \sum_{n=0}^\infty b_4(n,m)\, z^m\, q^n, 
$$
where $b_4(n,m)$ is the number of $4$-regular partitions of $n$ with $m$ parts,  we also have
$$
 G(-1,q) =\sum_{n=0}^\infty \big( b_{4,e}(n)-b_{4,o}(n) \big)\, q^n,$$ which establishes (ii).

%

\subsection{Combinatorial Proof} 

For the rest of this article we will use calligraphy style capital letters to denote the set of partitions enumerated by the  function denoted by the same letters. For example, $\mathcal P(n)$ denotes the set of partitions of $n$ and $\mathcal{POD}_e(n)$ is the set of partitions of $n$ into an even number of parts which are not congruent to $2$ modulo $4$. Recalling that $pod(n)$ has more than one partition theoretic interpretation, in the sequel we assume that 
$\mathcal{POD}(n)$ represents the set of partitions of $n$ having distinct odd parts. 
  
If $\lambda$ is a partition of $n$, we say that the size of $\lambda$ is $n$ and write $|\lambda|=n$.  The length of  $\lambda$,  denoted by $\ell(\lambda)$, is  the number of parts in $\lambda$.  We make the convention that $\lambda_{\ell(\lambda)+1}=0$.  Given two partitions $\lambda$ and $\mu$, we denote by $\lambda \cup \mu$ the partition whose parts are precisely the parts of $\lambda$ and  $\mu$ (with multiplicity). If each part of $\mu$ is also a part of $\lambda$ with equal or larger multiplicity, we denote by $\lambda \setminus \mu$ the partition obtained from $\lambda$ by removing the parts of $\mu$ (with multiplicity). 
Finally,  we will often write a partition $\lambda$  as $(\lambda^e, \lambda^o)$, where $\lambda^e$ (respectively $\lambda^o$)  consists of the even  (respectively odd) parts of $\lambda$. 

We start by proving the first identity of Theorem \ref{T:1}.
Denote by $\widetilde\Q_0(n)$ the set of partitions of $n$ into odd parts repeated no more than three times. 
Glaisher's transformation \cite{Gl}, which iteratively splits each even part of a distinct partition into two equal parts  until all parts are odd,  yields a bijection between $\Q_0(n)$ and $\widetilde\Q_0(n)$. 

For any partition $\lambda \in \widetilde\Q_0(n)$, we have $\ell(\lambda)\equiv n \pmod 2$. Thus, if $n$ is even (respectively odd), $\widetilde\Q_0(n)$ is a subset of  $\p_e(n)$ (respectively $\p_o(n)$).

Let $\A(n)=(\p_e(n)\cup \p_o(n))\setminus \widetilde\Q_0(n)$.  Inspired by \cite{G76}, we  define an involution $\varphi:\A(n)\to \A(n)$ that reverses the parity of the length of partitions. First, we introduce more useful notation. For any partition $\lambda$, if $d$ is a part of  $\lambda$  with  multiplicity $m_d$, we denote by $s_d$ the nonnegative integer satisfying $2^{s_d}\leq m_d< 2^{s_d+1}$. 
If $\lambda\in \A(n)$ has an odd part $d$ with $m_d\geq 4$, we let  $r_\lambda$ be the largest value of $2^{s_d}d$ among all such parts (else  we let $r_\lambda=0$). If $\lambda^e=\emptyset$, it follows from the definition of $\A(n)$ that $r_\lambda\neq 0$. We define $\varphi(\lambda)$ as follows. 
\begin{enumerate}
\item[(i)] If $r_\lambda\geq \lambda_1^e$, merge $2^{s_d}$ copies of $d$ into a new part  to obtain $\varphi(\lambda)$. 
\item[(ii)] If $r_\lambda<\lambda_1^e$, split $\lambda_1^e=2^{k_1}c_1$, where $k_1\geq 1$ and $c_1$ is odd,  into $2^{k_1}$ parts equal to $c_1$ obtain $\varphi(\lambda)$. 
\end{enumerate} Partitions from case (i) are mapped by $\varphi$ to partitions from case (ii) and vice-versa. Moreover, $\varphi$ is its own inverse and  reverses the parity of the length of partitions.  
 This finishes the proof of the first identity of Theorem \ref{T:1}.

For the second identity,
denote by $\widetilde\Q_2(n)$ the set of partitions $\lambda=(\lambda^e, \lambda^o)$ of $n$   such that $\lambda^o$ has distinct parts and each part of $\lambda^e$ is congruent to $2$ modulo $4$ and has even multiplicity. We use the following   variant of Glaisher's transformation to create a bijection between $\Q_2(n)$ and $\widetilde\Q_2(n)$: given $\eta\in \Q_2(n)$, 
split each  part of $\eta$ having the form $2^k c$, with $k \geqslant 2$ and $c$  odd, into $2^{k-1}$ parts equal to $2c$. The inverse of this transformation iteratively merges equal parts of a partition in  $\widetilde\Q_2(n)$ until all parts are distinct.  Since even parts of partitions in $\widetilde\Q_2(n)$ have even multiplicity, all obtained even  parts are divisible by $4$. 

For any partition $\lambda\in \widetilde\Q_2(n)$ we have $\ell(\lambda) \equiv n \pmod 2$. Thus, if $n$ is even  (respectively odd), $\widetilde\Q_2(n)$ is a subset of  $\B_{4,e}(n)$ (respectively $\B_{4.o}(n)$).

Let $\B(n)=(\B_{4,e}(n)\cup \B_{4,o}(n))\setminus \widetilde\Q_2(n)$. Notice that for every partition $(\lambda^e, \lambda^o)\in \B(n)$, if $\lambda^o$ has distinct parts, then there is at least one part in $\lambda^e$ with odd multiplicity.  We  define a transformation $\varepsilon:\B(n)\to \B(n)$ as follows. 
\begin{enumerate}
\item[(i)] If $\lambda$ has a part $4k+2$ with odd multiplicity with the property  that all parts less than $2k+1$ in $\lambda^o$ have multiplicity one, let $4a+2$ be the smallest such part and define $\varepsilon(\lambda)=(\lambda^e\setminus (4a+2), \lambda^o\cup (2a+1, 2a+1))$. 

\item[(ii)] Else, let $2c+1$ be the smallest repeated part in $\lambda^o$ and define $\varepsilon(\lambda)=(\lambda^e \cup (4c+2), \lambda^o\setminus(2c+1, 2c+1))$. 
\end{enumerate}
Then  $\varepsilon$ is an involution on $\B(n)$ that  maps a partition  from case (i) to a partition from case (ii) and vice-versa,  and changes the length of a partition by one.  This finishes the proof of the second identity of Theorem \ref{T:1}.

\section{Proof of Theorems \ref{T:2} and \ref{T:3}}
\label{S3}

\subsection{Analytic proofs}

\allowdisplaybreaks{
The following theta identities are often attributed to Gauss \cite[p.23, Eqs. (2.2.12), (2.2.13)]{Andrews98}:
\begin{equation}\label{eq:Gauss22}
1+2\sum_{n=1}^{\infty} (-1)^n\, q^{n^2} = \frac {(q;q)_\infty} {(-q;q)_\infty}
\end{equation}
and
\begin{equation}\label{eq:Gauss21}
\sum_{n=0}^{\infty} (-1)^{n(n+1)/2}\, q^{n(n+1)/2} = \frac {(q^2;q^2)_\infty} {(-q;q^2)_\infty}.
\end{equation}
The Jacobi triple product identity (cf.\ \cite[Eq.~(1.6.1)]{GaRaAA}) states that
\begin{equation} \label{eq:JTP} 
(z;q)_\infty (q/z;q)_\infty (q;q)_\infty=
\sum_{n=-\infty}^\infty (-z)^n\, q^{n(n-1)/2}.
\end{equation}

We can write
\begin{align}
(-q,-q^2,-q^3;q^4)_\infty\cdot \frac{(q^2;q^2)_\infty}{(-q;q^2)_\infty}
& = (-q^2;q^4)_\infty (q^2;q^2)_\infty\nonumber\\
& = (q^4;q^8)_\infty (q^4;q^4)_\infty\nonumber\\
& = \frac{(q^4;q^4)_\infty}{(-q^4;q^4)_\infty}\nonumber\\
& = 1+2\sum_{n=1}^{\infty} (-1)^n\, q^{4n^2},\label{eq:GFb}
\end{align}
and
\begin{align}
(-q,-q^3,-q^4;q^4)_\infty\cdot \frac{(q^2;q^2)_\infty}{(-q;q^2)_\infty} 
& = (-q^4;q^4)_\infty (q^2;q^2)_\infty \nonumber \\
& = \frac{(q^2;q^2)_\infty}{(q^4;q^8)_\infty} \nonumber \\
& = (q^2,q^6,q^8;q^8)_\infty \nonumber \\
& = \sum_{n=-\infty}^\infty (-1)^n\, q^{2n(2n-1)}\nonumber \\ 
& = \sum_{n=0}^\infty (-1)^{n(n+1)/2}\, q^{n(n+1)},\label{eq:GFa}
\end{align}
where we have invoked \eqref{eq:JTP} with $q$ replaced by $q^8$ and $z$ replaced by $q^2$.
Thus, we deduce that
\begin{align*}
\sum_{n=0}^\infty Q_0(n)\, q^n
& = \frac{(-q;q^2)_\infty}{(q^2;q^2)_\infty} \left( 1+2\sum_{n=1}^\infty (-1)^{n}\, q^{4n^2}\right)  \\
& = \left( \sum_{n=0}^{\infty} pod(n)\, q^n \right) \left( 1+2\sum_{n=1}^\infty (-1)^{n}\, q^{4n^2}\right) \\
& = \sum_{n=0}^{\infty} \left( pod(n)+2\sum_{k=1}^\infty (-1)^{k}\, pod\big(n-4k^2\big) \right) q^n,
\end{align*}
and
\begin{align*}
\sum_{n=0}^\infty Q_2(n)\, q^n
& = \frac{(-q;q^2)_\infty}{(q^2;q^2)_\infty} \sum_{n=0}^\infty (-1)^{n(n+1)/2}\, q^{n(n+1)} \\
& = \left( \sum_{n=0}^{\infty} pod(n)\, q^n \right) \left( \sum_{n=0}^\infty (-1)^{n(n+1)/2}\, q^{n(n+1)} \right) \\
& = \sum_{n=0}^{\infty} \left( \sum_{k=0}^\infty (-1)^{k(k+1)/2}\, pod\big(n-k(k+1)\big) \right) q^n,
\end{align*}
from which  Theorem \ref{T:2} follows.

Using \eqref{eq:Gauss21}, the relations \eqref{eq:GFb} and \eqref{eq:GFa} can be written as 
\begin{equation*}
\left( \sum_{n=0}^\infty Q_0(n)\, q^n \right) \left( \sum_{n=0}^{\infty} (-1)^{n(n+1)/2}\, q^{n(n+1)/2} \right) 
= 1+2\sum_{n=1}^\infty (-1)^{n}\, q^{4n^2}
\end{equation*}
and
\begin{equation*}
\left( \sum_{n=0}^\infty Q_2(n)\, q^n \right) \left( \sum_{n=0}^{\infty} (-1)^{n(n+1)/2}\, q^{n(n+1)/2} \right) 
= \sum_{n=0}^\infty (-1)^{n(n+1)/2}\, q^{n(n+1)}.
\end{equation*}
We rewrite these identities as 
$$
\sum_{n=0}^\infty \left(\sum_{k=0}^\infty (-1)^{k(k+1)/2}\, Q_0\big(n-k(k+1)/2\big)  \right) q^n
=  1+2\sum_{n=1}^\infty (-1)^{n}\, q^{4n^2}
$$
and
$$
\sum_{n=0}^\infty \left(\sum_{k=0}^\infty (-1)^{k(k+1)/2}\, Q_2\big(n-k(k+1)/2\big)  \right) q^n
=  \sum_{n=0}^\infty (-1)^{n(n+1)/2}\, q^{n(n+1)},
$$
from which Theorem \ref{T:3} follows. 
\subsection{Combinatorial proof of Theorem \ref{T:2} (ii)} 
First, we introduce a graphical representation of  partitions.  The Ferrers diagram of a partition $\lambda=(\lambda_1, \lambda_2, \ldots, \lambda_i)$ is an array of left justified boxes such that the $i$th row from the top contains $\lambda_i$ boxes. 

 If $n$ and $k$ are nonnegative integers, we define a transformation on $\p(n-k(k+1))$ as follows.  Start with  $(\lambda^e, \lambda^o)\in \p(n-k(k+1))$.  The $2$-modular Ferrers diagram of $\lambda^e$ is  obtained from the ordinary Ferrers diagram of $\lambda^e$ by replacing two boxes at a time in each row and placing a $2$ in the resulting box. To this diagram we append at the top the rotated Ferrers diagram of the staircase of length $k$ with each box filled with $2$. Next,  starting in the upper left corner of the obtained diagram, we draw a zig-zag line beginning with a right step and continuing with pairs of alternating down and right steps for as long as both step segments border boxes of the diagram.

For example,  if $\lambda^e=(14, 14, 12, 12,8,4)$ and $k=3$, we obtain 

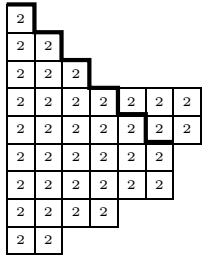
\begin{center}\tiny{\begin{tikzpicture}[inner sep=0in,outer sep=0in]\node (n) {\begin{varwidth}{0cm}{\begin{ytableau} 2 \\ 2&2\\ 2&2&2\\ 2&2&2&2&2&2&2 \\2&2&2&2&2&2&2\\2&2&2&2&2&2\\ 2&2&2&2&2&2\\ 2&2&2&2 \\ 2& 2 \end{ytableau}
}\end{varwidth}};
\draw[ultra thick] (-0,1.67)--(0.37,1.67)-- (0.37,1.3)--(.73,1.3)--(.73,.93)--(1.1,.93)--(1.1,.56)--(1.48,.56)--(1.48,.21)--(1.85,.21)--(1.85,-0.16)--(2.22,-0.16);
\end{tikzpicture}}\end{center}

The diagram defines two partitions into distinct even parts: the partition $\alpha$ whose $2$-modular Ferrers diagram is made up of  the columns below the zig-zag line, and the partition $\beta$ whose $2$-modular Ferrers diagram is formed by the rows to the right of the zig-zag line. In the example above, $\alpha=(18, 16, 12, 10, 6, 4)$ and $\beta=(6, 4)$.

The triple $(\lambda^o, \alpha, \beta)$ is completely determined by $(\lambda^e, \lambda^o)$ and $k$. Moreover $|\lambda^o|+|\alpha|+|\beta|=n$ and $k\leqslant \ell(\alpha)-\ell(\beta)\leqslant k+1$.  Denote by $\A_k(n)$ the set of triples of partitions $(\lambda^o, \alpha, \beta)$, where $\lambda^o$ is a partition into distinct odd parts, $\alpha$ and $\beta$ are partitions into distinct even parts, $|\lambda^o|+|\alpha|+|\beta|=n$, and $k\leqslant \ell(\alpha)-\ell(\beta)\leqslant k+1$. 
Then the transformation described above is  a bijection between $\p(n-k(k+1))$ and $\A_k(n)$ (for the inverse transformation we refer the reader to \cite[Section 2.1]{BM-mex}, where a similar transformation for ordinary Ferrers diagrams is defined). Thus, it remains to  show that 
\begin{equation}\label{q2a}\displaystyle{Q_2(n) = \sum_{k=0}^{\infty} (-1)^{k(k+1)/2}\, |\A_k(n)|}.\end{equation}

Let $\A(n)$ be the set of triples $(\lambda^o, \alpha, \beta)$, where $\lambda^o$ is a partition into distinct odd parts and $\alpha$ and $\beta$ are partitions into distinct even parts such that $|\lambda^o|+|\alpha|+|\beta|=n$ and $\ell(\alpha)- \ell(\beta)\geqslant 0$.

Each triple $(\lambda^o, \alpha, \beta)\in \A(n)$  with  $\ell(\alpha)>\ell(\beta)$ appears in exactly two of the sets $\A_k(n)$, namely when $k=\ell(\alpha)-\ell(\beta)$ and when $k=\ell(\alpha)-\ell(\beta)-1$. Triples $(\lambda^o, \alpha, \beta)\in \A(n)$ with $\ell(\alpha)=\ell(\beta)$ appear only in $\A_0(n)$.  Thus, from the parity of triangular numbers, the contribution of $(\lambda^o, \alpha, \beta)$ with $|\lambda^o|+|\alpha|+|\beta|=n$ to the righthand side of \eqref{q2a} is $$\begin{cases}  1  & \mbox{ if } \ell(\alpha)=\ell(\beta),\\  0 & \mbox{ if } \ell(\alpha)-\ell(\beta) \equiv 1 \pmod 2, \\   2 &  \mbox{ if } \ell(\alpha)-\ell(\beta) \equiv 0 \pmod 4, \\ -2 & \mbox{ if } \ell(\alpha)-\ell(\beta) \equiv 2 \pmod 4.\end{cases}$$

We denote by  $\M\A_0(n)$ be the multiset of triples $(\lambda^o, \alpha, \beta)\in \A(n)$ satisfying  $\alpha\neq \beta$, $\ell(\alpha)-\ell(\beta)\equiv 0\pmod 4$, and   if $\ell(\alpha)-\ell(\beta)>0$, the triple has multiplicity $2$ in $\M\A_0(n)$, and if $\ell(\alpha)-\ell(\beta)=0$, the triple  has multiplicity $1$. Similarly, we denote by  $\M\A_2(n)$ be the multiset of triples $(\lambda^o, \alpha, \beta)\in \A(n)$ with $\ell(\alpha)-\ell(\beta)\equiv 2\pmod 4$ and each triple has multiplicity $2$.

For any triple $(\lambda^o, \alpha, \beta)\in \A(n)$ let $i$ be the smallest positive integer such that $\alpha_i\neq \beta_i$ (we make the convention that $\beta_{\ell(\beta)+1}=0$). 

We now define a map from $\M\A_0(n)$ to $\M\A_2(n)$ as follows.
Start with $(\lambda^o, \alpha, \beta)\in\M\A_o(n)$ and suppose $\ell(\alpha)-\ell(\beta)=k$.

Case 1:  $\alpha_i<\beta_i$. Let $\tilde\alpha=\alpha \cup (\beta_i)$ and $\tilde \beta=\beta \setminus (\beta_i)$. Then, $\ell(\tilde\alpha)-\ell(\tilde\beta)=k+2$, $\tilde\alpha_i>\tilde\beta_i$, and   the first $i-1$ parts of $\tilde\alpha$ and $\tilde\beta$ are equal.  Moreover,  
 the triple $(\lambda^o,\tilde \alpha, \tilde\beta)$ lies in $\M\A_2(n)$. 

Case 2:   $\beta_i<\alpha_i$.  Let $\tilde\alpha=\alpha \setminus (\alpha_i)$ and $\tilde \beta=\beta \cup (\alpha_i)$. Then, $\ell(\tilde\alpha)-\ell(\tilde\beta)=k-2$, $\tilde\beta_i>\tilde\alpha_i$, and   the first $i-1$ parts of $\tilde\alpha$ and $\tilde\beta$ are equal.  
If $k>0$,  the triple $(\lambda^o,\tilde \alpha, \tilde\beta)$ lies in $\M\A_2(n)$.  If $k=0$, since $\ell(\tilde\beta)-\ell(\tilde\alpha)=2$, the triple   $(\lambda^o,\tilde \beta, \tilde\alpha)$ lies in $\M\A_2(n)$. Notice that this partition is also obtained from $(\lambda^o, \beta, \alpha)$ which is in Case 1. 

For the inverse of the transformation, start with $(\lambda^o, \tilde\alpha, \tilde\beta)\in\M\A_2(n)$ and suppose $\ell(\tilde\alpha)-\ell(\tilde\beta)=k$.

Case I: $\tilde\beta_i<\tilde\alpha_i$.  Let $\alpha=\tilde\alpha\cup(\tilde\beta_i)$ and $\beta=\tilde\beta\setminus(\tilde \beta_i)$. Then $\ell(\alpha)-\ell(\beta)=k+2$ and $(\lambda^o, \alpha, \beta)\in\M\A_o(n)$. 

Case II: $\tilde\alpha_i<\tilde\beta_i$.  Let $\alpha=\tilde\alpha\setminus(\tilde\alpha_i)$ and $\beta=\tilde\beta\cup(\tilde \alpha_i)$ Then $\ell(\alpha)-\ell(\beta)=k-2$ and $(\lambda^o, \alpha, \beta)\in\M\A_o(n)$. If $k=2$, one copy of $(\lambda^o, \tilde\alpha, \tilde\beta)$ is mapped to $(\lambda^o, \alpha, \beta)$ and the second copy is mapped to $(\lambda^o, \beta, \alpha)$.

Therefore, the transformation defined above is a bijection between $\M\A_0(n)$ and $\M\A_2(n)$. Then,  the righthand side of \eqref{q2a} is equal to the number of triples $(\lambda^o, \alpha, \beta)\in \A(n)$ with $\alpha=\beta$. These triples are in bijection with the partitions in $\Q_2(n)$ via the  mapping that takes $(\lambda^o, \alpha, \alpha)$ to $\lambda^o\cup 2\alpha$, where $2\alpha$ is the partition obtained from $\alpha$ by doubling each of its parts.

\section{Linear inequalities involving $pod(n)$}
\label{S4}

Linear inequalities involving partition functions have been studied extensively \cite{Andrews12,Andrews18,Guo,Katriel,Merca12,Merca16,Merca20,MK,MWY}.
For example, V.~J.~W.~Guo and J.~Zeng \cite{Guo} proved that
\begin{align*}
(-1)^{k-1} \sum_{j=0}^{2k-1} (-1)^{j(j+1)/2}\, pod\big(n-j(j+1)/2\big)\geqslant 0,
\end{align*}  for $n,k>0$.
Recently, G. E. Andrews and M. Merca \cite[Corollary 10]{Andrews18} established a 
partition theoretic interpretation of this inequality by showing that
\begin{align*}
(-1)^{k-1} \sum_{j=0}^{2k-1} (-1)^{j(j+1)/2}\, pod\big(n-j(j+1)/2\big) = MP_k(n),
\end{align*}
where $MP_k(n)$ is the number of partitions of $n$ in which the first part larger
than $2k-1$ is odd and appears exactly $k$ times and all other odd parts appear at
most once. For example, $MP_2(19)=10$, and the partitions in question are
\begin{align*}
& (9,9,1),\ (9,5,5),\ (8,5,5,1),\ (7,7,3,2),\ (7,7,2,2,1),\ (7,5,5,2),\\
& (6,5,5,3),\ (6,5,5,2,1),\ (5,5,3,2,2,2),\ (5,5,2,2,2,2,1). 
\end{align*}
Shortly after that, C. Ballantine, M. Merca, D. Passary and A. J. Yee \cite{BMPY} gave combinatorial proofs of this interpretation.

In this section, inspired by Theorem \ref{T:2}, we obtain new infinite families of linear inequalities for  $pod(n)$. To this end, we recall that 
an overpartition of $n$ is a nonincreasing sequence of natural numbers whose sum is $n$ in which the first occurrence of a number may be overlined \cite{Corteel}. For example, the eight  overpartitions of $3$ are
$$
(3),\ (\overline{3}),\ (2,1),\ (2,\overline{1}),\ (\overline{2},1),\ (\overline{2},\overline{1}),\ (1,1,1),\  (\overline{1},1,1).
$$
G. E. Andrews and M. Merca \cite{Andrews18} introduced the  function
$\overline{M}_k(n)$ which counts the number of overpartitions of $n$ in which the first part larger than $k$ appears at least $k+1$ times. For example, $\overline{M}_2(12)=16$, with the relevant overpartions being 
\begin{align*}
& (4,4,4),\,
(\overline{4},4,4),\,
(3,3,3,3),\,
(\overline{3},3,3,3),\,
(3,3,3,2,1),\,
(3,3,3,\overline{2},1),\\
& (3,3,3,2,\overline{1}),\,
(3,3,3,\overline{2},\overline{1}),\,
(\overline{3},3,3,2,1),\,
(\overline{3},3,3,\overline{2},1),\,
(\overline{3},3,3,2,\overline{1}),\\
& (\overline{3},3,3,\overline{2},\overline{1}),\,
(3,3,3,1,1,1),\,
(3,3,3,\overline{1},1,1),\,
(\overline{3},3,3,1,1,1),\,
(\overline{3},3,3,\overline{1},1,1).
\end{align*}
We now prove an identity that has Theorem \ref{T:2}.(i) as its limiting case when $k\to\infty$.

\begin{theorem}\label{T:4}
	For $n,k>0$, we have
	$$
	(-1)^k \left( pod(n)+2 \sum_{j=1}^{k} (-1)^{j}\, pod\big(n-4j^2\big) - Q_0(n) \right) 
	= \sum_{j=0}^{\left\lfloor n/4 \right\rfloor} Q_0(n-4j)\, \overline{M}_k(j).
	$$ 
\end{theorem}

\begin{proof}
		According to G. E. Andrews and M. Merca \cite[Theorem 7]{Andrews18}, we have  
	the following truncated version of  \eqref{eq:Gauss22}:
	\begin{align}
	& \frac{(-q;q)_{\infty}} {(q;q)_{\infty}} \left(1 + 2 \sum_{j=1}^{k} (-1)^j\, q^{j^2} \right)  \label{eq1} \\
	& \qquad = 1+ (-1)^k \frac{(-q;q)_k}{(q;q)_k} \sum_{j=k+1}^{\infty}
	\frac{2\,q^{j(k+1)}}{1-q^j}\cdot \frac{(-q^{j+1};q)_{\infty}}{(q^{j+1};q)_{\infty}}. \notag
	\end{align}	
	As explained in \cite[Proof of Corollary 8]{Andrews18}, the series on the right hand side  of this identity is the generating function for $\overline{M}_k(n)$, i.e.,
	$$
	\sum_{n=0}^\infty \overline{M}_k(n)\, q^n =  \frac{(-q;q)_k}{(q;q)_k} \sum_{j=k+1}^{\infty}
	\frac{2\,q^{j(k+1)}}{1-q^j}\cdot \frac{(-q^{j+1};q)_{\infty}}{(q^{j+1};q)_{\infty}}.
	$$
	Therefore, replacing $q$ by $q^4$ in  \eqref{eq1} yields
	\begin{align*}
	& \frac{(-q^4;q^4)_{\infty}} {(q^4;q^4)_{\infty}} \left(1 + 2 \sum_{j=1}^{k} (-1)^j\, q^{4j^2} \right) = 1+ (-1)^k \sum_{n=0}^\infty \overline{M}_k(n)\, q^{4n}.
	\end{align*}
	Multiplying both sides of this identity by 
	$(-q,-q^2,-q^3;q^4)_\infty$
	we obtain
	\begin{align*}
	& \frac{(-q;q^2)_{\infty}} {(q^2;q^2)_{\infty}} \left(1 + 2 \sum_{j=1}^{k} (-1)^j\, q^{4j^2} \right) - (-q,-q^2,-q^3;q^4)_\infty\\
	& \qquad\qquad = (-1)^k\, (-q,-q^2,-q^3;q^4)_\infty \sum_{n=0}^\infty \overline{M}_k(n)\, q^{4n}.
	\end{align*}	
Then, using \eqref{podgen} and \eqref{gfQ0}, the identity becomes
	\begin{align*}
& \left(\sum_{n=0}^\infty pod(n)\, q^n \right)  \left(1 + 2 \sum_{j=1}^{k} (-1)^j\, q^{4j^2} \right) - 
\sum_{n=0}^\infty Q_0(n)\, q^n\\
& \qquad\qquad = (-1)^k \left( \sum_{n=0}^\infty Q_0(n)\, q^n \right)  \left( \sum_{n=0}^\infty \overline{M}_k(n)\, q^{4n} \right) 
\end{align*}	
The  theorem follows by comparing coefficients of $q^n$ on the two sides
of this equation.
\end{proof}

As a consequence of Theorem \ref{T:4} we obtain the following infinite family of linear inequalities involving  $pod(n)$.

\begin{corollary}\label{C4.2}
	For $n,k>0$, we have
	$$
	(-1)^k \left( pod(n)+2 \sum_{j=1}^{k} (-1)^{j}\, pod\big(n-4j^2\big) - Q_0(n) \right) 
	\geqslant 0,
	$$ 
	with strict inequality if and only if $n\geqslant 4(k+1)^2$. \end{corollary}
\begin{remark}	For example,
	\begin{align*}
	& pod(n)-2pod(n-4) \leqslant Q_0(n),\\
	& pod(n)-2pod(n-4)+2pod(n-16) \geqslant Q_0(n),\\
	& pod(n)-2pod(n-4)+2pod(n-16)-2pod(n-36) \leqslant Q_0(n),\\
	& pod(n)-2pod(n-4)+2pod(n-16)-2pod(n-36)+2pod(n-64) \geqslant Q_0(n).\\
	\end{align*}
	\end{remark}
Our next result has  Theorem \ref{T:2}.(ii) as its limiting case when $k\to\infty$.

\begin{theorem}\label{T:5}
	For $n,k>0$, we have
	$$
	(-1)^{k} \left(Q_2(n) - \sum_{j=0}^{2k-1} (-1)^{j(j+1)/2}\, pod\big(n-j(j+1)\big) \right) 
	= \sum_{j=0}^{\left\lfloor n/2 \right\rfloor} Q_2(n-2j)\, MP_k(j).
	$$ 
\end{theorem}

\begin{proof}
			According to G. E. Andrews and M. Merca \cite[Theorem 9]{Andrews18}, we have  
	the following truncated version of  \eqref{eq:Gauss21}:
	\begin{align}
	& \frac{(-q;q^2)_{\infty}} {(q^2;q^2)_{\infty}} \sum_{j=0}^{2k-1} (-q)^{j(j+1)/2}   \notag\\
	& \qquad = 1+ (-1)^{k-1} \frac{(-q;q^2)_k}{(q^2;q^2)_{k-1}} \sum_{j=k}^{\infty}
	\frac{q^{k(2j+1)}\, (-q^{2j+3};q^2)_\infty}{(q^{2j+2};q^2)_\infty}. \label{eq1a}
	\end{align}	
	As explained in \cite[Proof of Corollary 10]{Andrews18}, the series on the right hand side  of this identity is the generating function for $MP_k(n)$, i.e.,
	$$
	\sum_{n=0}^\infty MP_k(n)\, q^n =  \frac{(-q;q^2)_k}{(q^2;q^2)_{k-1}} \sum_{j=k}^{\infty}
	\frac{q^{k(2j+1)}\,(-q^{2j+3};q^2)_\infty}{(q^{2j+2};q^2)_\infty}.
	$$
	Therefore, replacing $q$ by $q^2$ in   \eqref{eq1a} yields
	\begin{align*}
	& \frac{(-q^2;q^4)_{\infty}} {(q^4;q^4)_{\infty}} \sum_{j=0}^{2k-1} (-1)^{j(j+1)/2}\, q^{j(j+1)}  
	 = 1+ (-1)^{k-1} \sum_{n=0}^\infty MP_k(n)\, q^{2n}.
	\end{align*}
	Multiplying both sides of this identity by 
	$(-q,-q^3,-q^4;q^4)_\infty$
	we obtain
	\begin{align*}
	& \frac{(-q;q^2)_{\infty}} {(q^2;q^2)_{\infty}} \sum_{j=0}^{2k-1} (-1)^{j(j+1)/2}\, q^{j(j+1)}  
	-(-q,-q^3,-q^4;q^4)_\infty\\
	&\qquad\qquad = (-1)^{k-1}\, (-q,-q^3,-q^4;q^4)_\infty \sum_{n=0}^\infty MP_k(n)\, q^{2n}.
	\end{align*}
	Then, using \eqref{podgen} and \eqref{gfQ1}, the identity becomes
	\begin{align*}
	& \left( \sum_{n=0}^\infty pod(n)\, q^n \right)  \left( \sum_{j=0}^{2k-1} (-1)^{j(j+1)/2}\, q^{j(j+1)}  \right) 
	- \sum_{n=0}^\infty Q_2(n)\, q^n \\
	&\qquad\qquad = (-1)^{k-1} \left(  \sum_{n=0}^\infty Q_2(n)\, q^n \right) \left(  \sum_{n=0}^\infty MP_k(n)\, q^{2n} \right).
	\end{align*}
	The  theorem follows by comparing coefficients of $q^n$ on the two sides
of this equation.
\end{proof}

As a consequence of Theorem \ref{T:5} we obtain the following infinite family of linear inequalities involving  $pod(n)$.

\begin{corollary}\label{C4.4}
	For $n,k>0$, we have
	$$
	(-1)^{k} \left( Q_2(n) - \sum_{j=0}^{2k-1} (-1)^{j(j+1)/2}\, pod\big(n-j(j+1)\big)  \right) 
	\geqslant 0,
	$$ 
	with strict inequality if and only if $n\geqslant 2k(2k+1)$.\end{corollary}
	\begin{remark} For example,
	\begin{align*}
	& pod(n)-pod(n-2) \geqslant Q_2(n),\\
	& pod(n)-pod(n-2)-pod(n-6)+pod(n-12) \leqslant Q_2(n),\\
	& pod(n)-pod(n-2)-pod(n-6)+pod(n-12)\\
	& \qquad\qquad\qquad\qquad+pod(n-20)-pod(n-30) \geqslant Q_2(n),\\
	& pod(n)-pod(n-2)-pod(n-6)+pod(n-12)\\
	& \qquad\qquad+pod(n-20)-pod(n-30)-pod(42)+pod(56) \leqslant Q_2(n).
	\end{align*}
	\end{remark}

\section{Ramanujan type congruences} \label{RC}

In recent years, many congruences for the number of $\ell$-regular partitions have been discovered
by G.~E.~Andrews, M.~D.~Hirschhorn and J.~A.~Sellers \cite{Andrews10},
S.~C.~Chen \cite{Chen},
S.-P.~Cui and N.~S.~S.~Gu \cite{Cui,Cui14}, 
B.~Dandurand and D.~Penniston \cite{Dand}, 
D.~Furcy and D.~Penniston  \cite{Furcy}, 
B.~Gordon and K.~Ono \cite{Gordon}, 
W.~J.~Keith \cite{Keith}, 
B.~L.~S.~Lin and A.~Y.~Z.~Wang \cite{Lin}, 
J.~Lovejoy and D.~Penniston \cite{Lovejoy}, 
D.~Penniston \cite{Penn,Penn8}, 
E.~X.~W.~Xia \cite{Xia14b,Xia}, 
E.~X.~W.~Xia and O.~X.~M.~Yao \cite{Xia14,Xia14a}, 
O.~X.~M.~Yao \cite{Yao},
and J.~J.~Webb \cite{Webb}.

For  example, for $\alpha\geqslant 1$ and $n\geqslant 0$,  from \cite{Andrews10}  we have
\begin{align*}
& b_4\left( 3^{2\alpha+1} n + \frac{17\cdot 3^{2\alpha}-1}{8}\right) \equiv 0 \pmod 2, \\
& b_4\left( 3^{2\alpha+2} n + \frac{11\cdot 3^{2\alpha+1}-1}{8}\right) \equiv 0 \pmod 2, \\
& b_4\left( 3^{2\alpha+2} n + \frac{19\cdot 3^{2\alpha+1}-1}{8}\right) \equiv 0 \pmod 2,
\end{align*}
and for $r\in\{ 13,21,29,37\}$,  from \cite{Chen}  we have 
$$
b_4\left( 5^{2\alpha+2} n + \frac{r\cdot 5^{2\alpha+1}-1}{8}\right) \equiv 0 \pmod 4.
$$
From  \cite{Xia14b} ,  for  $\alpha\geqslant 0$, $n\geqslant 0$ and $r\in\{11,19\}$, we have 
$$
b_4\left( 3^{4\alpha+4} n + \frac{r\cdot 3^{4\alpha+3}-1}{8}\right) \equiv 0 \pmod 8.
$$

To facilitate the study of  partition functions, S. Radu \cite{Radu1a,Radu1}
considered a class of  functions $a(n)$ defined by
\begin{align}
\sum_{n=0}^\infty a(n)\,q^n = \prod_{\delta |M} (q^{\delta};q^{\delta})_\infty^{r_\delta},\label{R1}
\end{align}
where the product is over the positive divisors of $M>0$ and 
$r_\delta\in \mathbb Z$. Using  the ideas of H.~Rademacher \cite{Rademacher}, M.~Newman \cite{Newman1,Newman2} and O.~Kolberg \cite{Kolberg},  Radu \cite{Radu1a} discovered a method for establishing 
 congruences of the form
$$a(mn+t) \equiv 0 \pmod{u},$$
for fixed $m$, $t$ and $u$, and any $n \geqslant 0$, 
and developed the so-called Ramanujan–Kolberg algorithm
 \cite{Radu1} for deriving identities involving the generating functions of $a(mn + t)$ and modular functions for $\Gamma_0(N)$
(a description of this algorithm can be found in P.~Paule and S.~Radu \cite{Paule}).
Very recently, N.~A.~Smoot \cite{Smoot} provided a successful Mathematica implementation of Radu's algorithm. The package is called \texttt{RaduRK} and requires \texttt{4ti2}, a software package for algebraic, geometric and combinatorial problems on linear spaces. Instructions for the proper installation of these packages can be found in \cite{Smoot}. 

The procedure
$$\texttt{RK[N,M,r,m,j]}$$
take as input an  integer $N\geqslant 2$ which defines the congruence subgroup $\Gamma_0(N)$, a generating function defined by $M$ and $r=(r_\delta)_{\delta | M}$ as in \eqref{R1}, and an arithmetic progression $mn+j$ with $0\leqslant j < m$. The algorithm decides if there exists an identity of the form
$$
f(q)\cdot \prod_{j'\in P}\left(  \sum_{n=0}^\infty a(mn+j')\,q^n \right) = \sum_{g\in A} g\cdot p_g(t),
$$
where
$$A = \{1,g_1,\ldots,g_{u} \}\qquad\text{and}\qquad
\{p_g(t)\}_{g\in A} = \{p_1,p_{g_1},\ldots,p_{g_{u}}\}.$$
We remark that $t,g_1,\ldots,g_{u}$ are modular function for  $\Gamma_0(N)$.
For the definition of these notions  
and  a general introduction to \texttt{RaduRK} algorithm, see \cite{Paule}. 
For the correctness proof and details of the algorithm, see \cite{Smoot}.


In this section, we use the implemented \texttt{RaduRK} algorithm to prove several congruences modulo $16$  for $b_4(n)$.   
Letting $M=4$ and  $r_1=-1, r_2=0, r_4=1$ in \eqref{R1} yields the generating function for $b_4(n)$.
 
\begin{theorem}\label{TH5.1}
	Let $\alpha\in\{8,13,18,23\}$.	For all $n\geqslant 0$, we have
	$$b_4(25n+\alpha) \equiv 0\pmod {16}.$$
\end{theorem}
Since,   $b_4(8)=16$,  $b_4(25+13)=8528$,  $b_4(18)=208$ and  $b_4(23)=592$, we have that 
$$
\sum_{n=0}^\infty b_4(25n+\alpha)\,q^n \not\equiv 0 \pmod{32},
$$ for all $\alpha\in\{8,13,18,23\}$.
Thus, Theorem \ref{TH5.1} follows directly from the following two lemmas.

\begin{lemma}\label{L1}
	\begin{align*}
	& \left( \sum_{n=0}^\infty b_4(25n+8)\,q^n \right) \left( \sum_{n=0}^\infty b_4(25n+23)\,q^n \right) 
	\equiv 0 \pmod {256}.
	\end{align*}		
\end{lemma}

\begin{proof}
  To establish this congruence identity, we consider the \texttt{RaduRK} program with
  $$\texttt{RK[20,4,\{-1,0,1\},25,8]}.$$

The algorithm returns:
\begin{align*}
P & = \{8,23\} \\
f(q) &= \frac{(q;q)_{\infty }^{52}\, (q^4;q^4)_{\infty }^{12}\, (q^{10};q^{10})_\infty^{32}} 
{q^{31}\, (q^2;q^2)_\infty^{32}\, (q^{5};q^{5})_{\infty}^{10}\, (q^{20};q^{20})_{\infty}^{54}}\\
t &= \frac{ (q^4;q^4)^4_{\infty}\, (q^{10};q^{10})^2_{\infty}}
{q^2\, (q^2;q^2)^2_{\infty}\, (q^{20};q^{20})^4_{\infty}}\\
A &= \left\{1,\frac{(q^4;q^4)_\infty\, (q^{5};q^{5})_\infty^5}{q^3\, (q;q)_\infty\, (q^{20};q^{20})_\infty^5}-\frac{(q^4;q^4)_\infty^{4}\, (q^{10};q^{10})_\infty^2}{q^2\, (q^2;q^2)_\infty^{2}\, (q^{20};q^{20})_\infty^4}\right\}\\
\{p_g(t)\}_{g\in A} &= \left\{-387500000000\, t^2 - 4722300000000\, t^3 + 19755240000000\, t^4 \right.\\
&\qquad -13492968000000\, t^5 - 28902996000000\, t^6 + 51723282400000\, t^7 \\
&\qquad -27746680960000\, t^8 - 697717120000\, t^9 + 7036326368000\, t^{10}\\
&\qquad -2875996422400\, t^{11} + 195120171520\, t^{12} + 113806525952\, t^{13}\\
&\qquad +2380696832\, t^{14} + 2340096\, t^{15},\\
&\qquad  12500000000\, t +262500000000\, t^2 + 1449800000000\, t^3 \\
&\qquad - 9119240000000\, t^4+13179468000000\, t^5 - 2546388000000\, t^6\\ 
&\qquad -9465334400000\, t^7 +8751301760000\, t^8 - 2237298720000\, t^9\\
&\qquad - 628733856000\, t^{10}+318550950400\, t^{11} + 22754516480\, t^{12} \\
&\qquad \left.  + 119739648\, t^{13} + 9472\, t^{14} \right\}.
\end{align*}

Taking into account that $256$ is the common factor of all the coefficients of the polynomials $p_g(t)$, we deduce the  identity
\allowdisplaybreaks{
\begin{align*}
& \frac{(q;q)_{\infty }^{52}\, (q^4;q^4)_{\infty }^{12}\, (q^{10};q^{10})_\infty^{32}} 
{q^{31}\, (q^2;q^2)_\infty^{32}\, (q^{5};q^{5})_{\infty}^{10}\, (q^{20};q^{20})_{\infty}^{54}}  \left( \sum_{n=0}^\infty b_4(25n+8)\,q^n \right) \left( \sum_{n=0}^\infty b_4(25n+23)\,q^n \right)\\
& = 256\, Y_1+256\, Y_2 \left(  \frac{(q^4;q^4)_\infty\, (q^{5};q^{5})_\infty^5} {q^3\, (q;q)_\infty\, (q^{20};q^{20})_\infty^5}-\frac{(q^4;q^4)_\infty^{4}\, (q^{10};q^{10})_\infty^2}{q^2\, (q^2;q^2)_\infty^{2}\, (q^{20};q^{20})_\infty^4} \right),
\end{align*}
where
\begin{align*}
Y_1 &= -1513671875\, t^2 - 18446484375\, t^3 + 77168906250\, t^4  \\
& \quad -52706906250\, t^5 - 112902328125\, t^6 + 202044071875\, t^7\\
& \quad -108385472500\, t^8 - 2725457500\, t^9 + 27485649875\, t^{10}\\
& \quad -11234361025\, t^{11} + 762188170\, t^{12} + 444556742\, t^{13}\\
& \quad +9299597\, t^{14} + 9141\, t^{15},\\
Y_2 &= 48828125\, t +1025390625\, t^2 + 5663281250\, t^3 - 35622031250\, t^4\\
& \quad +51482296875\, t^5 - 9946828125\, t^6 -36973962500\, t^7\\
& \quad +34184772500\, t^8 - 8739448125\, t^9 - 2455991625\, t^{10} \\
& \quad +1244339650\, t^{11} + 88884830\, t^{12} + 467733\, t^{13} + 37\, t^{14}.
\end{align*}
Lemma \ref{L1} follows immediately.
}
\end{proof}

\begin{lemma}\label{L2}
	\begin{align*}
	& \left( \sum_{n=0}^\infty b_4(25n+13)\,q^n \right) \left( \sum_{n=0}^\infty b_4(25n+18)\,q^n \right) 
	\equiv 0 \pmod {256}.
	\end{align*}		
\end{lemma}

\begin{proof}
	
	To establish this congruence identity, we consider the \texttt{RaduRK} program with
	$$\texttt{RK[20,4,\{-1,0,1\},25,13]}.$$

The algorithm returns:
\begin{align*}
P & = \{13,18\} \\
f(q) &= \frac{(q;q)_{\infty }^{52}\, (q^4;q^4)_{\infty }^{12}\, (q^{10};q^{10})_\infty^{32}} 
{q^{31}\, (q^2;q^2)_\infty^{32}\, (q^{5};q^{5})_{\infty}^{10}\, (q^{20};q^{20})_{\infty}^{54}}\\
t &= \frac{ (q^4;q^4)^4_{\infty}\, (q^{10};q^{10})^2_{\infty}}
{q^2\, (q^2;q^2)^2_{\infty}\, (q^{20};q^{20})^4_{\infty}}\\
A &= \left\{1,\frac{(q^4;q^4)_\infty\, (q^{5};q^{5})_\infty^5}{q^3\, (q;q)_\infty\, (q^{20};q^{20})_\infty^5}-\frac{(q^4;q^4)_\infty^{4}\, (q^{10};q^{10})_\infty^2}{q^2\, (q^2;q^2)_\infty^{2}\, (q^{20};q^{20})_\infty^4}\right\}\\
\{p_g(t)\}_{g\in A} &= \left\{62500000000\, t - 650000000000\, t^2 - 4309800000000\, t^3 \right.\\
& \quad 	+19463740000000\, t^4 - 13372968000000\, t^5 - 29049996000000\, t^6 \\
& \quad 	+51949994400000\, t^7 - 27941584960000\, t^8 - 595662560000\, t^9\\
& \quad 	+7002545088000\, t^{10} - 2869711052800\, t^{11} + 194775499520\, t^{12}\\
& \quad 	+113785213952\, t^{13} + 2380097792\, t^{14} + 2273536\, t^{15},\\
& \quad     -50000000000\, t + 587500000000\, t^2 + 712300000000\, t^3\\
& \quad    -8147740000000\, t^4 + 12342468000000\, t^5 - 2047288000000\, t^6\\
& \quad -9669166400000\, t^7 + 8801213760000\, t^8 - 2240988480000\, t^9\\
& \quad  -629908896000\, t^{10} + 318735916800\, t^{11} + 22753946880\, t^{12}\\
& \quad \left.  + 120139008\, t^{13} + 13312\, t^{14} \right\}.
\end{align*}

Taking into account that $256$ is the common factor of all the coefficients of the polynomials $p_g(t)$, we deduce the  identity
			\allowdisplaybreaks{
		\begin{align*}
		& \frac{(q;q)_{\infty}^{52}\, (q^4;q^4)_\infty^{12}\, (q^{10};q^{10})_\infty^{32}} 
		  {q^{31}\, (q^2;q^2)_\infty^{32}\, (q^{5};q^{5})_\infty^{10}\, (q^{20};q^{20})_\infty^{54}} 
		  \left( \sum_{n=0}^\infty b_4(25n+13)\,q^n \right) \left( \sum_{n=0}^\infty b_4(25n+18)\,q^n \right)\\
		& = 256\, Y_1+256\, Y_2 \left(  \frac{(q^4;q^4)_\infty\, (q^{5};q^{5})_\infty^5} {q^3\, (q;q)_\infty\, (q^{20};q^{20})_\infty^5}-\frac{(q^4;q^4)_\infty^{4}\, (q^{10};q^{10})_\infty^2}{q^2\, (q^2;q^2)_\infty^{2}\, (q^{20};q^{20})_\infty^4} \right)  ,
		\end{align*}
		where
		\begin{align*}
		Y_1 &= 244140625\, t -2539062500\, t^2 -16835156250\, t^3  \\
		& \quad +76030234375\, t^4 -52238156250\, t^5 -113476546875\, t^6\\
		& \quad +202929665625\, t^7 -109146816250\, t^8 -2326806875\, t^{9}\\
		& \quad +27353691750\, t^{10} -11209808800\, t^{11} +760841795\, t^{12}\\
		& \quad +444473492\, t^{13} + 9297257\, t^{14}+8881\, t^{15},\\
		Y_2 &= -195312500\, t +2294921875\, t^2 + 2782421875\, t^3 - 31827109375\, t^4\\
		& \quad +48212765625\, t^5 - 7997218750\, t^6 -37770181250\, t^7\\
		& \quad +34379741250\, t^8 - 8753861250\, t^9 - 2460581625\, t^{10} \\
		& \quad +1245062175\, t^{11} + 88882605\, t^{12} + 469293\, t^{13} + 52\, t^{14}.
		\end{align*}
Lemma \ref{L2} follows immediately.
	}
\end{proof}

The \texttt{RaduRK} algorithm can also be used to introduce congruences modulo $64$  for $b_4(n)$.

\begin{theorem}\label{TH5.2}
	Let $\alpha\in\{13,20,27,34,41,48\}$.	For all $n\geqslant 0$, we have
	$$b_4(49n+\alpha) \equiv 0\pmod {64}.$$
\end{theorem}

Theorem \ref{TH5.2} follows directly from the following lemma.

\begin{lemma}\label{L3} The following congruences hold. 
	\begin{enumerate}
		\item [(i)] $\displaystyle{\prod\limits_{\alpha \in \{13,20,34\}} \sum\limits_{n=0}^\infty b_4(49n+\alpha)\,q^n \equiv 0 \pmod {64^3}}$
		\item[(ii)] $\displaystyle{\prod\limits_{\alpha \in \{27,41,48\}} \sum\limits_{n=0}^\infty b_4(49n+\alpha)\,q^n \equiv 0 \pmod {64^3}}$
	\end{enumerate}
\end{lemma}

The proof of this lemma is quite similar to the proof of Lemmas \ref{L1} and \ref{L2}, so we
omit the details. In the algorithm, we use $$\texttt{RK[28,4,\{-1,0,1\},49,13]}$$ and $$\texttt{RK[28,4,\{-1,0,1\},49,27]}.$$

We end this section by noting that it is likely that there exist infinite families of congruences modulo $16$ or modulo $64$ for $b_4(n)$. However, we were unable to find more congruences due to the running time of the \texttt{RaduRK} program.

\section{Open problems and concluding remarks}
\label{S5}

In this paper, inspired by the decompositions of $Q_0(n)$ and $Q_2(n)$
in terms of the partition function $pod(n)$ given by Theorem \ref{T:2}, we  derived  new infinite families of linear inequalities involving  $pod(n)$.
Here we show that there is another way to decompose $Q_0(n)$ using $pod(n)$, and  also establish decompositions of $Q_2(n)$ and $b_4(n)$ in terms of the function $\overline{p}(n)$ which enumerates the overpartitions of $n$ (see \cite{Corteel}).

\begin{theorem}\label{T:6}
	For $n\geqslant 0$, we have
	\begin{align}\label{watsonw0}
	Q_0(n) =  \sum_{k=0}^\infty pod\left(\frac{n-k(k+1)/2}{2} \right).
	\end{align}
\end{theorem}

\begin{proof}[Analytic proof]
	By \eqref{gfQ0} we have 
	$$\sum_{n=0}^\infty Q_0(n)\, q^n =(-q,-q^2,-q^3;q^4)_\infty
	 = (-q,-q^3,q^4;q^4)_\infty \cdot \frac{(-q^2;q^4)_\infty}{(q^4;q^4)_\infty}. $$ Next, using \eqref{podgen} and replacing $q$  by $q^4$ and $z$ by $-q$ in \eqref{eq:JTP}, we find
	\begin{align*}
	\sum_{n=0}^\infty Q_0(n)\, q^n 
	& = \left( \sum_{n=-\infty}^\infty q^{2n^2-n}\right) \left( \sum_{n=0}^\infty pod(n)\, q^{2n} \right) \\
	& = \left( \sum_{n=0}^\infty q^{n(n+1)/2}\right) \left( \sum_{n=0}^\infty pod(n)\, q^{2n} \right) \\
	& = \sum_{n=0}^\infty \left( \sum_{k=0}^\infty pod\Big(\frac{n-k(k+1)/2}{2} \Big) \right) q^n.
	\end{align*}
\end{proof}

\begin{proof}[Combinatorial proof]

We begin by recalling a result due to G.~N.~Watson \cite{Watson}.  If we denote by $Q_{odd}(n)$ the number of partitions of $n$ into distinct odd parts, then \begin{equation}\label{watsonqodd}
Q_{odd}(n)=\sum_{k=0}^\infty p\left(\frac{n-k(k+1)/2}{4} \right).
\end{equation}

A combinatorial proof of \eqref{watsonqodd} is given in  \cite{W08} using abacus displays, which were first introduced in \cite{James}. For the convenience of the reader, we briefly describe the construction here, using the language of \cite[section 2]{Schmidt}. 

Given a partition $\lambda$, fill the boxes of its Ferrers diagram with alternating $0$s and $1$s starting with $0$ in the upper left corner. Draw horizontal lines through rows ending in $0$ and vertical lines  through columns ending in $1$. The boxes at the intersection of the lines (pushed toward the northwest) form the  Ferrers diagram of a partition $\alpha$. Repeat the process with horizontal lines through rows ending in $1$ and vertical lines  through columns ending in $0$ to obtain a partition $\beta$. The pair $(\alpha, \beta)$ is referred to as the $2$-quotient of $\lambda$. The $2$-core  of $\lambda$ is the staircase partition $\delta$ whose Ferrers diagram is obtained from the diagram of $\lambda$ by removing $2$-hooks (either two boxes in a row or two boxes in a column) such that after every removal the obtained diagram is the Ferrers diagram of a partition. Then $$|\lambda|=|\delta|+2|\alpha|+2|\beta|.$$  It is shown in \cite[Theorem 2.7.30]{James} that the triple $(\delta, \alpha, \beta)$ completely determines $\lambda$. Clearly, if $\lambda$ is self-conjugate, then $\alpha=\beta$ and $(\delta, \alpha)$ determines $\lambda$. 

Given a partition $\mu\in \Q_{odd}(n)$, consider the partitions $\mu^{sc}$ whose Ferrers diagram is obtained from that of $\mu$ by straightening the hooks nested along the diagonal (this is the classical bijection proving that the number of partitions of $n$ into distinct odd parts equals the number of self-conjugate partitions of $n$). Define $\xi(\mu)$ to be the partition $\alpha$ in the $2$-quotient $(\alpha, \alpha)$ of $\mu^{sc}$.  Then the transformation $$\xi: \mathcal Q_{odd}(n)\to\bigcup_{k=0}^\infty \mathcal P\left(\frac{n-k(k+1)/2}{4} \right)$$ is a bijection. 

To prove \eqref{watsonw0}, start with a partition $\lambda=(\lambda^e, \lambda^o) \in \Q_0(n)$. The parts of $\lambda^e$ are congruent to $2$ modulo $4$. Let $\tilde{\lambda}^e$ be the partitions whose parts are the parts of $\lambda^e$ divided by $2$. Thus, $\tilde \lambda^e$ is a partition with distinct odd parts. We denote by $\tilde \lambda^o$  the partition whose parts are the parts of $\xi(\lambda^o)$ multiplied by $2$. Let $\psi(\lambda)=\tilde \lambda^o \cup \tilde \lambda^e$. Then $$\psi: \Q_0(n) \to  \bigcup_{k=0}^\infty \p\left(\frac{n-k(k+1)/2}{2}\right)$$ is a bijection.





\end{proof}

In a similar way,  $Q_2(n)$ and  $b_4(n)$ can be expressed in terms of $\overline{p}(n)$. 
\begin{theorem}\label{T:7}
	For $n\geqslant 0$, we have
	\begin{align*}
	Q_2(n) =  \sum_{k=0}^\infty \overline{p}\left(\frac{n-k(k+1)/2}{4} \right),
	\end{align*}
	and
	\begin{align*}
	b_4(n) =  \sum_{k=0}^\infty \overline{p}\left(\frac{n-k(k+1)/2}{2} \right).
	\end{align*}
\end{theorem}

\begin{proof}[Analytic proof]
	Arguing as in the proof of Theorem \ref{T:6}, we have
	\begin{align*}
	\sum_{n=0}^\infty Q_2(n)\, q^n 
	& = (-q,-q^3,q^4;q^4)_\infty \cdot \frac{(-q^4;q^4)_\infty}{(q^4;q^4)_\infty}\\
	& = \left( \sum_{n=0}^\infty q^{n(n+1)/2}\right) \left( \sum_{n=0}^\infty \overline{p}(n)\, q^{4n} \right) \\
	& = \sum_{n=0}^\infty \left( \sum_{k=0}^\infty \overline{p}\Big(\frac{n-k(k+1)/2}{4} \Big) \right) q^n
	\end{align*}
	and
	\begin{align*}
	\sum_{n=0}^\infty b_4(n)\, q^n 
	& = \frac{1}{(q;q^2)_\infty (q^2;q^4)_\infty} =  (-q;q)_\infty (-q^2;q^2)_\infty\\
	& = (-q,-q^3;q^4)_\infty \cdot \frac{(-q^2;q^2)_\infty}{(q^2;q^4)_\infty} \\
	& = (-q,-q^3,q^4;q^4)_\infty \cdot \frac{(-q^2;q^2)_\infty}{(q^2;q^2)_\infty}\\
	& = \left( \sum_{n=0}^\infty q^{n(n+1)/2}\right) \left( \sum_{n=0}^\infty \overline{p}(n)\, q^{2n} \right) \\
	& = \sum_{n=0}^\infty \left( \sum_{k=0}^\infty \overline{p}\Big(\frac{n-k(k+1)/2}{2} \Big) \right) q^n.
	\end{align*}
\end{proof}

\begin{proof}[Combinatorial proof]

We start with a partition $\lambda=(\lambda^e,\lambda^o) \in \Q_2(n)$. The parts of $\lambda^e$ are congruent to $0$ modulo $4$. Let $\bar \lambda^e$ be the overpartition whose parts are the parts of $\lambda^e$ divided by $4$ and each part is overlined.  Let $\rho(\lambda)=\xi(\lambda^o) \cup \bar \lambda^e$. Then $$\rho: \Q_2(n) \to  \bigcup_{k=0}^\infty \overline \P\left(\frac{n-k(k+1)/2}{4}\right)$$ is a bijection. 

To prove the second identity we make use of the Fu-Tang combinatorial proof  \cite{Fu-Tang} of another identity due to G.~N.~Watson \cite{Watson}.  If we denote by $Q(n)$ the number of partitions of $n$ into distinct  parts, then the Fu-Tang bijection $$\zeta: \Q(n)\to\bigcup_{k=0}^\infty \P\left(\frac{n-k(k+1)/2}{2} \right)$$ 
(see Section 2 of  \cite{BM} for succinct description of $\zeta$) yields
\begin{equation}\label{watsonq}
Q(n)=\sum_{k=0}^\infty p\left(\frac{n-k(k+1)/2}{2} \right). 
\end{equation}

Start with $\lambda=(\lambda^e, \lambda^o) \in \B_4(n)$. The parts of $\lambda^e$ are congruent to $2$ modulo $4$. We divide each part of $\lambda^e$ by $2$ to obtain a partition $\tilde \lambda^e$ into  odd parts. We use Glaisher's bijection to transform $\tilde \lambda^e$ into a partition with distinct parts whose parts we overline  to obtain an overpartition $\bar \lambda^e$. 
Next, we apply Glaisher's bijection to $\lambda^o$ to obtain a partition $\tilde \lambda^o$ into distinct parts and let $\chi(\lambda)=\zeta(\tilde \lambda^o)\cup \bar \lambda^e$. Then the function  $$\chi : \B_4(n) \to  \bigcup_{k=0}^\infty \overline \P\left(\frac{n-k(k+1)/2}{2}\right)$$ is a bijection. 
\end{proof}

We note that Theorems \ref{T:6} and \ref{T:7} provide new identities of Watson type (more on such  identities  can be found in \cite{BM}).

Numerical evidence suggests that $Q_0(n)$ and $Q_2(n)$ satisfy the following linear homogeneous inequalities analogous to those given by Corollaries \ref{C4.2} and \ref{C4.4}. 


\begin{conjecture}
	For $n,k\geqslant 0$, we have
	\begin{align*}
	& (-1)^{k-1} \left( \sum_{j=0}^{2k-1} (-1)^{j(j+1)/2}\, Q_0\big(n-j(j+1)/2\big)
	- \xi_n \right) \geqslant 0,
	\end{align*}
	with strict inequality if and only if $n\geqslant k(2k+1)$. \end{conjecture}
	\begin{remark} For example, 
	\begin{align*}
	& Q_0(n) - Q_0(n-1) \geqslant \xi_n,\\
	& Q_0(n) - Q_0(n-1) - Q_0(n-3)+Q_0(n-6) \leqslant \xi_n,\\
	& Q_0(n) - Q_0(n-1) - Q_0(n-3)+Q_0(n-6)+Q_0(n-10)-Q_0(n-15) \geqslant \xi_n.
	\end{align*}\end{remark}

\begin{conjecture}
	For $n,k\geqslant 0$, we have
	\begin{align*}
	& (-1)^{k-1} \left( \sum_{j=0}^{2k-1} (-1)^{j(j+1)/2}\, Q_2\big(n-j(j+1)/2\big)
	- \chi_n \right) \geqslant 0,
	\end{align*}
	with strict inequality if and only if $n\geqslant k(2k+1)$. \end{conjecture}
	\begin{remark} For example, 
	\begin{align*}
	& Q_2(n) - Q_2(n-1) \geqslant \chi_n,\\
	& Q_2(n) - Q_2(n-1) - Q_2(n-3)+Q_2(n-6) \leqslant \chi_n,\\
	& Q_2(n) - Q_2(n-1) - Q_2(n-3)+Q_2(n-6)+Q_2(n-10)-Q_2(n-15) \geqslant \chi_n.
	\end{align*}\end{remark}

\section*{Acknowledgements} The authors thank an anonymous referee for many useful suggestions that helped improve the presentation of the article.

\bigskip


\end{document}